\documentclass[preprint]{imsart}

\usepackage{latexsym, amssymb, amscd, amsthm, amsxtra, amsmath,amsthm,mathtools }
\RequirePackage{natbib}
\RequirePackage[colorlinks,citecolor=blue,urlcolor=blue]{hyperref}
\usepackage{graphics, graphicx, color}
\usepackage{ifpdf}
\usepackage{multirow}

\graphicspath{{./images/}}

\startlocaldefs

\newtheorem{thm}{Theorem}
\newtheorem{cor}[thm]{Corollary}
\newtheorem{lem}[thm]{Lemma}
\newtheorem{prop}[thm]{Proposition}
\theoremstyle{definition}
\newtheorem{defn}{Definition}
\newtheorem*{defn*}{Definition}
\theoremstyle{remark}
\newtheorem{remark}{Remark}

\newcommand{\norm}[1]{\|#1\|}
\newcommand{\abs}[1]{\left\vert#1\right\vert}

\newcommand{\Real}{\mathbb R}

\newcommand{\To}{\longrightarrow}

\def\argmin{\mathop{\rm argmin}}
\def\Cov{\mbox{Cov}}

\def\E{\mathbb{E}}

\def\diag{\mbox{diag}}

\def\r{\mathbf{r}}
\def\v{\mathbf{v}}
\def\u{\mathbf{u}}

\def\av{\mathbf a}

\def\cv{\mathbf c}

\def\ev{\mathbf e}

\def\rv{\mathbf r}

\def\vv{\mathbf v}

\def\xv{\mathbf x}
\def\yv{\mathbf y}

\def\Xv{\mathbf X}

\def\1v{\mathbf 1}
\def\0v{\mathbf 0}

\endlocaldefs

\begin{document}

\begin{frontmatter}

\title{Inference on subspheres model for directional data}
\runtitle{CLT for subsphere}

\author{\fnms{Sungkyu} \snm{Jung}\ead[label=e2]{sungkyu@pitt.edu}}
\address{Department of Statistics, University of Pittsburgh, Pittsburgh, PA 15222, U.S.A.\\ \printead{e2}}



\runauthor{S. Jung}

\begin{abstract}
Modeling deformations of a real object is an important task in computer vision, biomedical engineering and biomechanics. In this paper, we focus on a situation where a three-dimensional object is rotationally deformed about a fixed axis, and assume that many independent observations are available. Such a problem is generalized to an estimation of concentric, co-dimension 1, subspheres of a polysphere. We formulate least-square estimators as    generalized Fr\'{e}chet means, and evaluate the consistency and asymptotic normality.
\end{abstract}

%



\end{frontmatter}

\section{Introduction}

This work is motivated by the study of rotational deformation of 3D objects. \cite{Schulz2012} proposed an estimation of rotational axis for 3D bodies whose deformation is modelled by directional vectors. The method of  \cite{Schulz2012} can be understood as fitting concentric circles.
Directional data in 3D lie in the unit sphere $S^{m} = \{ \xv \in \Real^{m+1}: \norm{\xv}_2 = 1 \}$ with $m = 2$. When a set of direction vectors is rotated by a rotation operator, the trajectories of the rotation form concentric circles.

A circle on $S^2$ is a set of equidistance points, parameterized by a center $\cv \in S^2$ and geodesic radius $r \in (0,\pi)$, and is $$[\cv,r] = \{\xv \in S^2 : \xv'\cv = \cos(r) \}.$$
The $K$-set of concentric circles is a collection of circles on $S^2$ with a common center $\cv$ and is
$$[\cv,\rv] = \{ (\xv_1,\ldots,\xv_K) \in (S^2)^K : \xv_j'\cv = \cos(r_j), j = 1,\ldots,K \},$$
for $\rv = (r_1,\ldots,r_K) \in (0, \pi)^K$. Figure \ref{fig:illustcircle} illustrates an example of $[\cv,r]$ and $[\cv,\rv]$.

For a data set $\{ \Xv_1,\ldots, \Xv_n\}$, $\Xv_i = (\xv_{i1},\ldots,\xv_{iK}) \in (S^2)^K$, the method in \cite{Schulz2012} fits $[\cv,\rv]$ by minimizing the sum of squared residuals. Let $\rho(\Xv_i, [\cv,\rv])$ measure the residual, then the estimate is
\begin{equation}\label{eq:estimation}
 [\cv_n,\rv_n] = \argmin_{[\cv,\rv]} \sum_{i=1}^n \rho^2(\Xv_i, [\cv,\rv]).
\end{equation}
The estimator is a generalized sample Fr\'{e}chet mean \citep{Huckemann2011}. An example of the fir $[\cv_n,\rv_n]$ in the special cases of $K = 1$ (left panel) and $K = 4$ (right panel) is plotted in Fig.~\ref{fig:illustcircle}.
Now let $[\cv_0,\rv_0]$ be the population version defined as,
\begin{equation}\label{eq:minimizerModel}
[\cv_0,\rv_0] = \argmin_{[\cv,\rv]} E \rho^2(\Xv, [\cv,\rv]).
\end{equation}

In this work, we investigate the large sample behavior of the estimator. In particular, we show that $ [\cv_n,\rv_n]$ is consistent estimator of $[\cv_0,\rv_0]$ and also that $ [\cv_n,\rv_n]$ is asymptotically normal, as summarized in the following proposition.

\begin{prop}
Suppose all assumptions in Theorem~\ref{thm:consistencyandnormality} are satisfied. Assume in adition that $[\hat\mu_n] = [\cv_n,\rv_n]$ and $[\mu_0] = [\cv_0,\rv_0]$ of (\ref{eq:estimation}) and (\ref{eq:minimizerModel}) exist and are unique. Then for a metric $d$ defined later in Section~\ref{sec:defineP},
\begin{enumerate}
\item[i)] $[\hat\mu_n]$ is strongly consistency with $[\mu_0]$ in the sense that
\begin{equation*}
    \lim_{n\to\infty} d([\hat\mu_n],[\mu_0])  = 0 \ \mbox{ with probability 1},
\end{equation*}
\item[ii)] For a mapping $\phi$ from the space of $[\hat\mu_n]$ to a vector space, i.e., $\phi([\hat\mu_n]) \in \Real^{\nu}$, with $\nu = m+K$, there exists a $\nu \times \nu$ matrix $A_\phi$ and a $\nu \times \nu$ covariance matrix $\Sigma_\phi$  such that
 \begin{equation*}
\sqrt{n} A_\phi \{\phi(\mu_n) -  \phi(\mu)\} \To N_\nu(0, \Sigma_\phi) \mbox{  in distribution as } n \to \infty.
\end{equation*}
\end{enumerate}
\end{prop}

%
%

\begin{figure}[tb!]
\centering
\ifpdf
\vskip -1cm
      \includegraphics[width=1\textwidth]{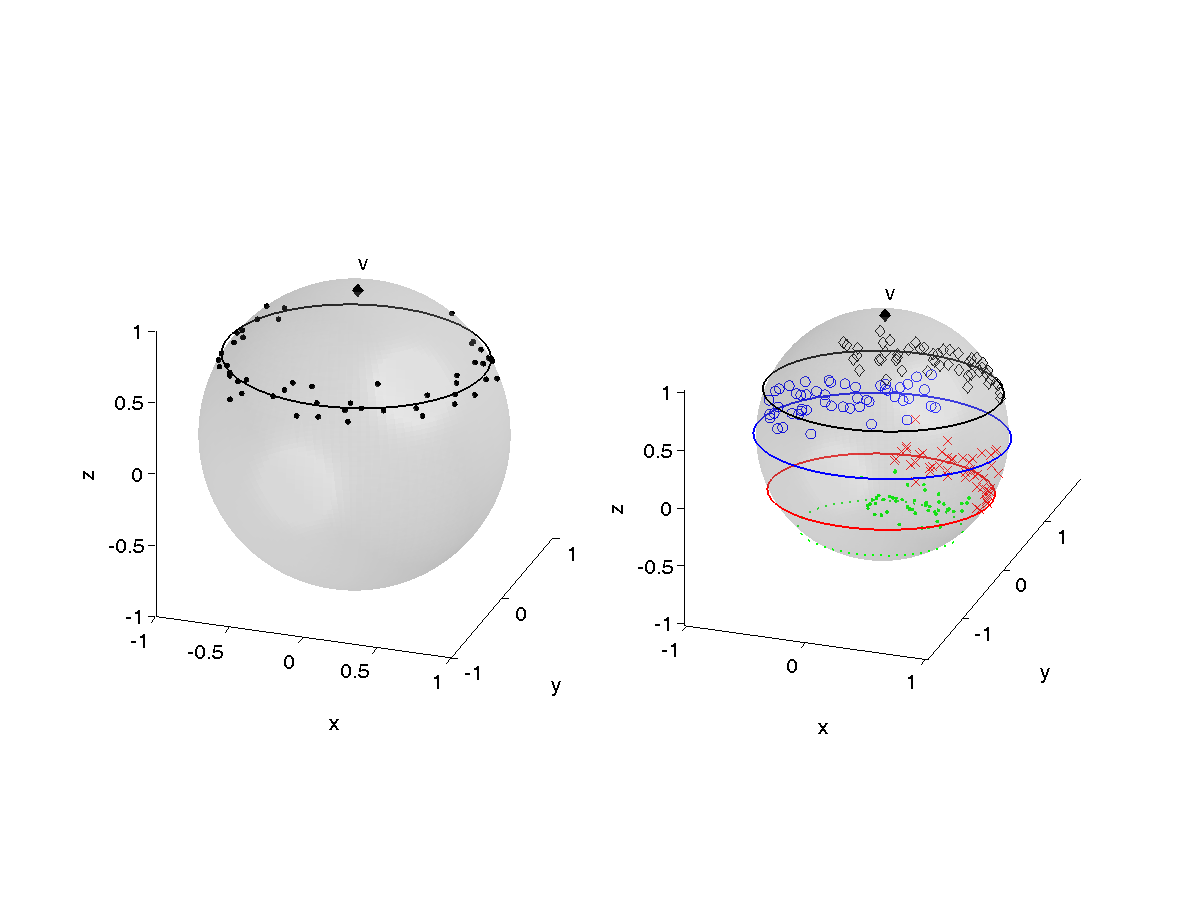}
\vskip -1cm
\else
      \includegraphics[width=1\textwidth]{illustcircle.ps}
\fi
  \caption{A circle on $S^2$ (left) and a set of concentric circles (right)}\label{fig:illustcircle}
\end{figure}

The problem of fitting $[\cv,r]$ on $S^2$ can be generalized to fitting a subsphere $[\cv,r]$, $\cv \in S^{m}$, on an $m$-dimensional unit sphere, $m \ge 2$. Such a problem is relevant to a backward dimension reduction of directional and shape data \citep{JungPNS}. We present our results in the general $m$-sphere case.

The implication of our work lies in providing the large sample confidence interval for rotation axes $\cv$ and a large sample hypothesis test for rotation axes. Suppose we have estimates of $A_\phi$ and $\Sigma_\phi$. Let $\hat{A}_\phi \to A_\phi$ and $\hat\Sigma_\phi \to \Sigma_\phi$ in probability as $n \to \infty$. Then by Slutsky's theorem and Cramer-Wold device, we still have asymptotic normality.

Since the estimate $\cv$ of the axis of rotation is on a curved surface $S^{m}$, we rely on the asymptotic theory developed for Fr\'{e}chet mean on general pseudo-metric spaces, following \cite{Ziezold1977,Bhattacharya2003,Huckemann2011,Huckemann2011CLT}.

\section{Fitting concentric circles}
The problem of fitting concentric circles (or subspheres) for $\Xv_1, \ldots, \Xv_n \in (S^{m})^K$, $\Xv_i = (\xv_{i1},\ldots,\xv_{iK})$, can be formulated as a general form of optimization problem that minimizes
$$
F_{L}(\cv,\rv) = \frac{1}{nK}\sum_{i=1}^n \sum_{j=1}^K  L(\cv, r_j | \xv_{ij}), \quad \cv \in S^{m}, 0 < r_j \le \pi/2,
$$
where $L$ is a loss function. Different forms of the loss function $L$ leads to different notions of residuals from $[\cv,\rv]$ to $\xv \in S^{m}$. The loss functions considered here are relevant to squared distances on $S^{m}$. The original optimization problem of \cite{Schulz2012} is given by the squared geodesic distance $L(\cv,r|x) = \{\arccos(\langle \cv, x \rangle ) - r\}^2$.

We now give a list of loss functions we consider. In preparation, we define the following. See also the illustration in Fig.~\ref{fig:AllDistances}. The geodesic distance between two points $\xv,\yv \in S^{m}$ is defined by the arc length of the shortest geodesic segment connecting $\xv$ and $\yv$ and is
$$\rho_I(\xv,\yv) = \arccos(\xv'\yv) \in [0, \pi].$$
The geodesic distance is often called the intrinsic distance. On the other hand, the extrinsic distance is defined by the Euclidean distance between $\xv$ and $\yv$ in the embedding $\Real^{m+1}$ of $S^{m}$ and is
$$\rho_E (\xv,\yv) = \norm{ \xv - \yv}_2 \in [0, 2].$$
The distance between a point $\xv$ and a set $A \subset S^{m}$ is defined by the shortest distance between $\xv$ and members of $A$,
$$ \rho(\xv, A) = \inf_{\yv \in A} \rho_I(\xv, \yv), \ \rho_E(\xv, A) = \inf_{\yv \in A} \rho_E(\xv, \yv).$$
When the set $A$ is a subsphere, $A = [\cv,r]$, precise expressions for the distances are available through the projection.
The projection $P_{[\cv,r]}\xv$ of $\xv$ onto $[\cv,r]$ is the point in $[\cv,r]$ given by
\begin{align}
P_{[\cv,r]}\xv & = \argmin_{\yv \in [\cv,r] } \rho_I(\xv, \yv)  \nonumber \\
               & = \argmin_{\yv \in [\cv,r] } \rho_E(\xv, \yv) \nonumber \\
               & = \cos(r) \cv + \sin(r) \av, \nonumber
\end{align}
where $\av = ( \cv - (\cv^T\xv)\xv) / \norm{\cv - (\cv^T\xv)\xv}$. Therefore, we have
$\rho_I(\xv, [\cv,r]) = \rho_I(\xv,P_{[\cv,r]}\xv) =  \abs{\rho_I(\xv,\cv) - r} = \abs{ \arccos(\xv^T\cv) - r},$
and
$\rho_E(\xv, [\cv,r]) = \norm{ x - P_{[\cv,r]}x} $.

We consider intrinsic loss, extrinsic loss, slicing loss, and naive extrinsic loss functions, summarized in Table~\ref{tab:lossfunctions}. These loss functions are all squared distances, which we discuss next. See also Fig.~\ref{fig:AllDistances}.

\begin{table}[tb]
\centering
\begin{tabular}{ccl}
Name  & Symbol & Associated distance \\
\hline
Intrinsic squared loss & $L_I = \rho_I^2 $   &  $\rho_I(\xv,[\cv,r]) = \abs{\rho_I(\xv,\cv) - r}$ \\
Extrinsic squared loss & $L_E = \rho_E^2$  &  $\rho_E(\xv,[\cv,r]) = \norm{ x - P_{[\cv,r]}x}_2$\\
Naive extrinsic squared loss & $L_{N} = \rho_N^2$     &  $\rho_N(\xv, [\cv,r]) = \abs{ \norm{\xv - \cv}_2 - r_E}$\\
Slicing squared loss   & $ L_{S} = \rho_S^2$& $\rho_S(\xv, [\cv,r]) = \frac{1}{2}\abs{ \norm{\xv - \cv}^2_2 - r^2_E} $
\end{tabular}
\caption{List of loss functions and their associated distance functions. Here $r_E = {2} \sin(r / 2)$.}
\label{tab:lossfunctions}
\end{table}

When the residuals are measured by the geodesic distance, we have \emph{intrinsic squared loss} function $L_I(\cv,r | \xv) = \rho_I^2(\xv, [\cv,r])$, which is the squared intrinsic distance between $\xv$ and its projection on $[\cv,r]$. The \emph{extrinsic squared loss} function is obtained when the residuals are measured by the extrinsic distance, and is
\begin{eqnarray*}
 L_E(\cv,r | \xv) = \rho_E(\xv, [\cv,r])^2
  &=&  2 \sin \{ \frac{\rho_I(\xv ,P_{[\cv,r]}\xv)}{2} \} \\
  & = &1- 2\cos ( \rho_I(\xv,\cv) - r) \\
  &=&  2 - 2 ( \xv'\cv \cos(r) + \sqrt{1 - (\xv'\cv)^2} \sin(r)).
\end{eqnarray*}
The extrinsic and intrinsic loss functions are closely related, by $L_E(\cv,r | \xv) =  L_I(\cv,r | \xv)/2 + O(  L_I(\cv,r | \xv)^2 ) \}$.
Next, consider a rank $m$ hyperplain $V$ in $\Real^{m+1}$ spanned by elements of $[\cv,r] \subset \Real^{m+1}$. Then $V = \{\yv \in \Real^{m+1} : \cv^T \yv   - \cos(r) = 0\}$. A useful view of the subsphere fitting is  understanding the fit $[\cv,r]$ as a slicing of the sphere $S^{m}$ by the hyperplain $V$. This leads to a simple definition of residual. Denote $\xv^\prime$ the orthogonally projected $\xv$ onto the affine hyperplain. We have $ \xv^\prime =  \xv - (\cv^T\xv - \cos(r)) \cv$, which leads to the definition of \emph{slicing squared loss}
$$ L_{S}(\cv,r | \xv) = \norm{ \xv- \xv^\prime}^2 = (\cv^T\xv  - \cos(r))^2,$$
which is understood as the residual of $\xv$ in the embedded space $\Real^{m+1}$ when slicing the sphere with the affine hyperplain $V$.
Lastly, \emph{naive extrinsic squared loss} is given by replacing the geodesic distance $\rho_I$ with $\rho_E$ in $L_I(\cv,r | \xv) = (\rho_I(\xv,\cv) - r)^2$. Also replacing $r$ with its extrinsic counterpart $r_E = {2} \sin(r / 2)$, the naive extrinsic squared loss is
$$L_{N}(\cv,r | \xv) =  (\norm{\cv - \xv} - {2} \sin(r / 2) )^2.$$
One can further alternate $L_N$ by measuring the square of squared extrinsic distances, i.e.,
$L_0(\cv,r | \xv) =  (\norm{\cv - \xv}^2 - r_E^2 )^2.$ It turns out that the loss function $L_0$ is equivalent to the slicing squared loss, that is, $ L_S(\cv,r | \xv) = \norm{\xv - \xv^\prime}^2 = (\cv^T\xv  - \cos(r))^2 = \frac{1}{4}L_0(\cv,r | \xv), $
which in turn leads to a definition of slicing distance
$$\rho_S(\xv,[\cv,r]) = \frac{1}{2} \abs{\norm{\cv - \xv}^2 - r_E^2} = \frac{1}{2} (\norm{\cv - \xv} - r_E^2)(\norm{\cv - \xv} + r_E^2).$$
\begin{figure}[tb!]
\centering
\ifpdf
      \includegraphics[width=1\textwidth]{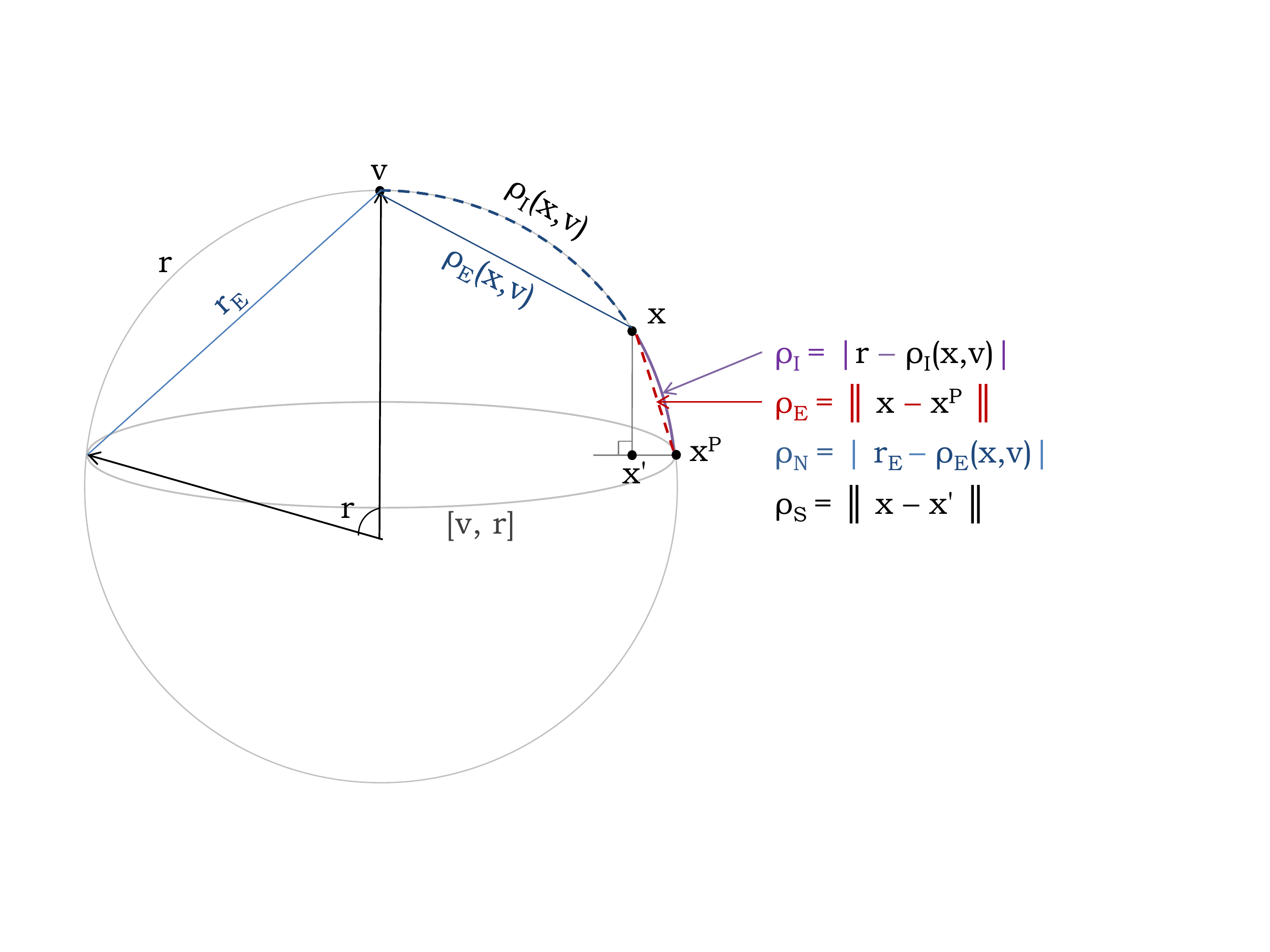}
      \vskip -4em
\else
\fi
  \caption{All distances illustrated.}\label{fig:AllDistances}
\end{figure}

\begin{remark}
The distance functions used in the above loss functions are all \emph{equivalent} in the sense that whenever $\rho_{\imath}(x,y) \to 0$, $\rho_{\jmath}(x,y) \to 0$.
\end{remark}

\begin{remark}
When the slicing squared loss is used, the optimization task reduces to the usual eigenvalue problem. Specifically, consider the following minimization problem:
$$
F(\cv,r) = \frac{1}{n}\sum_{j=1}^n L_S(\cv,r | \xv_j) = \frac{1}{n}\sum_{j=1}^n (\cv^T\xv_j  - \cos(r))^2,
$$
for $(\cv,r) $ satisfying $\cv^T\cv = 1$, $r \in (0,\pi)$.
Denote for simplicity $b = \cos(r)$. Then using the $ (m+1) \times n$ matrix of data points $\mathbf{X} = [\xv_1, \ldots, \xv_n]$, we have
$$ F(\cv,b) = (\cv^T \mathbf{X} - b 1_n^T)(\cv^T \mathbf{X} - b 1_n^T)^T,$$
where $1_n$ denotes the column vector of size $n$ with all elements 1.
Given any $v$, we have the minimizer $\hat{b} =  \frac{1}{n}\sum_{j=1}^n v^Tx_j = \frac{1}{n}v^T\mathbf{X} 1_n$, which leads to
$$ F_1(v) = v^T\mathbf{X}( I_n - \frac{1}{n}1_{n}1_{n}^T)\mathbf{X}^T v := v^T S_X v,$$
where $S_X$ is the sample covariance matrix of the embedded $\xv_j$. Notice that the minimizer $\hat{\cv}$ of $F_1$ with the constraint $\cv^T\cv = 1$ is the same as the eigenvector of $S_X$ corresponding to the smallest eigenvalue. The solution $\hat{\cv}$ also satisfies $\hat{\cv}^T\hat{\cv} = 1$ and thus leading to $\hat{b} = \frac{1}{n}\sum_{j=1}^n \hat{\cv}^T\xv_j \in [-1, 1]$.

Moreover, when the problem is restricted to the case $r = \pi/2$, the corresponding solution $\cv$ is also obtained by a eigen-decomposition. This is because $r = \pi/2$ gives $b = 0$, and thus the corresponding $\hat{\cv}$ is the eigenvector of $ \mathbf{X} \mathbf{X}^T$ corresponding to the smallest eigenvalue.
\end{remark}

We can now reformulate the estimation procedure using the distance functions. Let $\rho_{S^{m}}$ be either $\rho_I$, $\rho_E$, $\rho_S$ or $\rho_N$.
Then the general problem we consider is to minimize the following function over $\cv \in S^{m}, 0 < r_j \le \pi/2$,
\begin{eqnarray}\label{eq:problem}
F_{\rho_{S^{m}}}(\cv,\rv) = \frac{1}{nK}\sum_{i=1}^n \sum_{j=1}^K  \rho_{S^{m}}^2(\cv, r_j | \xv_{ij}),
\end{eqnarray}
or to find a minimizer $[\cv,\rv]$ of $F_{\rho_{S^{m}}}([\cv,\rv]) := F_{\rho_{S^{m}}}(\cv,\rv)$ among a collection of subspheres.

In the next section we evaluate the asymptotic properties of estimate, $[\cv_n,\rv_n] = \argmin F_{\rho_{S^{m}}}([\cv,\rv])$, compared to the population counterpart $[\cv_0,\rv_0]$.

\begin{remark}
We have four distances (related to the loss functions): Intrinsic (geodesic) distance, extrinsic distance, slicing distance and naive extrinsic distance. Among these only slicing distance function $\rho_I(\xv,[\cv,r])$ is smooth in the second argument for all $\cv \in S^{m}$. Other choices are not. As shown in Fig.~\ref{fig:Illustdistances}, other distance functions are not smooth at $\cv = \xv$ or at $-\xv$. In such a case we will make a special assumption about the random variable $\Xv \in (S^{m})^K$.
\begin{itemize}
    \item[(A1)] There exists $\epsilon > 0 $ such that $P( \Xv \in \cup_{j=1}^K \{ X = (\xv_1,\ldots,\xv_K)  : \rho_I(\xv_j, \cv_0) < \epsilon \mbox{ or } \rho_I(\xv_j, -\cv_0) < \epsilon, \xv_j \in S^{m} \} ) = 0$.
\end{itemize}
In other words, we assume that there is no observation near the true axis $\cv_0$.
Then all distance functions are smooth in the second argument for $\cv$ in the $\epsilon$-neighborhood of $\cv_0$, which leads to the asymptotic normality at $\mu = [\cv_0, r_0]$.
\begin{figure}[tb!]
\centering
\ifpdf
      \includegraphics[width=0.7\textwidth]{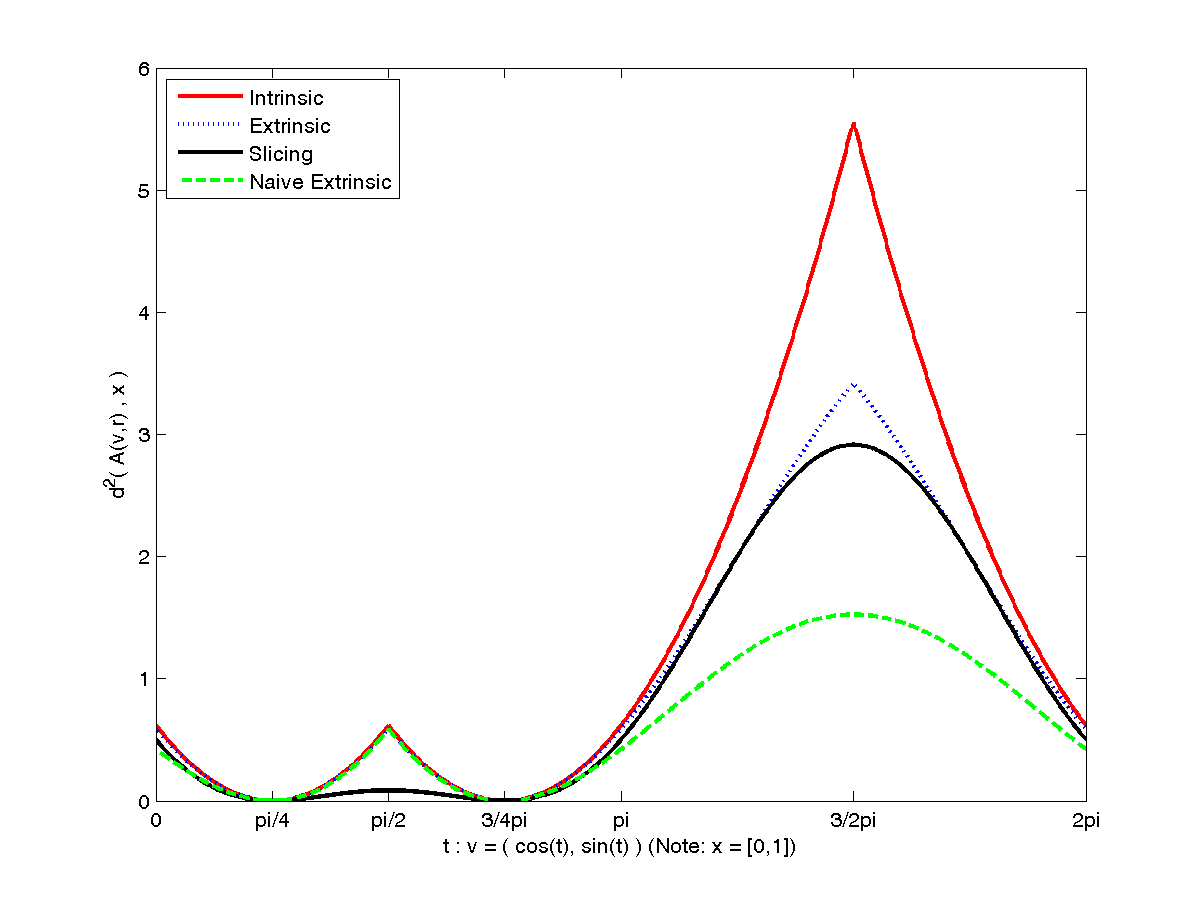}
\else
      \includegraphics[width=0.7\textwidth]{Illustdistances}

\fi
  \caption{The slicing distance function is smooth. Squared distances for different values of $\cv$, with fixed $r  = \pi/4$ and $x = [0,1]'$ are plotted here.}\label{fig:Illustdistances}
\end{figure}
\end{remark}

\section{Main Results}

In this section, we study large sample behaviors of the least-squares estimator (\ref{eq:estimation}). In particular, a consistency of the estimator $(\hat{\cv}, \hat{\r})$ with the population Fr\'{e}chet $\rho$-mean $(\cv,\r)$ defined in (\ref{eq:minimizerModel}) and an asymptotic normality of the estimator will be evaluated.

Let $\Xv,\Xv_1,\Xv_2,\ldots$ be i.i.d. random elements on the product of unit $d$-sphere $(S^{m})^K$. Each random element $\Xv$ or $\Xv_i$ is a mapping from some probability space $(\Omega, \mathcal{F}, \mathcal{P})$ to $(S^{m})^K$ equipped with its Borel $\sigma$-field. A distance function $\rho_{S^{m}}$ naturally leads to the product metric $\rho$ defined on $(S^{m})^K$,
$\rho^2(\Xv, [\cv, \rv]) = \frac{1}{K} \sum_{j = 1}^K \rho_{S^{m}}^2(\xv_j, [\cv,r_j] )$ for $\Xv = (\xv_1,\ldots,\xv_K) \in (S^{m})^K$.
%
%
%
%
%



The parameter pair $(\cv,\r)$ represents a $K$-set of concentric circles in $S^2$ (or concentric spheres when $d > 2$).
 Note that $(\cv,\r)$ and $(-\cv,\pi-\r)$ represent the same set of concentric circles, i.e. $[\cv,\r] = [-\cv, \pi-\r]$. Here we have used a convention that $\pi-\r = (\pi-r_1,\ldots, \pi-r_K)'$. In order to provide a convergence of   $(\hat{\cv}, \hat{\r})$ to $(\cv,\r)$, a distance function between two concentric circles will be first defined.

\subsection{The set of $[\cv,\rv]$}\label{sec:defineP}
We begin with a definition of $P= P(m,K)$, the set of all concentric circles (or spheres) in $S^m$  for any fixed $2 \le m < \infty$, $K \in \mathbb{N}$. Recall that a circle in the unit 2-sphere may be identified with a center $\cv \in S^2$ and a radius $r_j \in (0,\pi)$. Likewise, a subsphere in the unit $m$-sphere is represented by a pair $(\cv,r_j)$. Denote by $P_0$ the space of the center-radii pair $(\cv,\rv)$ as
\begin{equation*}
  P_0 = \{ (\cv,\rv) | c \in S^m,  r_j \in (0,\pi), j = 1,\ldots,K\} = S^m \times (0,\pi)^K,
\end{equation*}
whose dimension is $m+K$.
Since both $(\cv,\rv)$ and $(-\cv, \pi-\r)$ represent the same concentric circles, define an equivalence relation $\sim$ such that for $(\cv_1,\r^1), (\cv_2,\r^2) \in P_0$,
$(\cv_1,\r^1) \sim (\cv_2,\r^2)$ if and only if $(\cv_2,\r^2) \in [\cv_1,\r^1] = \{ (\cv_1,\r^1), (-\cv_1, \pi-\r^1) \}$, where $[\cv_1,\r^1]$ denotes the equivalence class of $(\cv_1,\r^1)$.  Then $P$ is defined as a quotient set of $P_0$ with respect to the binary relation $\sim$, i.e., $P = \{[\cv,\r] | (\cv,\rv) \in P_0 \}=P_0/ \sim$. For any $(\cv,\rv) \in  P_0$, $[\cv,\r] \in P$, and the dimension of $P$ is $m+K$.

We construct $d: P \times P \to [0,\infty)$ as a distance function.
For $[p_1]=[\cv_1,\r^1], [p_2] = [\cv_2,\r^2] \in P$, let $$d([p_1],[p_2]) = \min\{d_1(p_1,p_2),d_2(p_1,p_2)\},$$
where
\begin{eqnarray*}
 d_1(p_1,p_2) &=& \sqrt{\arccos(\cv_1'\cv_2)^2 + \| \r^1-\r^2\|^2 }, \\
 d_2(p_1,p_2) &=& \sqrt{\arccos(-\cv_1'\cv_2)^2 + \| \pi-\r^1-\r^2\|^2 }.
\end{eqnarray*}
We can also define $d: P_0 \times P_0 \to [0 ,\infty)$, $d(p_1,p_2) = d([p_1],[p_2])$, as a distance function on $P_0$.
Note that $\arccos(\cv_1'\cv_2) = d_g(\cv_1,\cv_2)$ is the minimal angle to rotate $\cv_2$ onto $\cv_1$, and, for the $j$th circle, $\abs{r^1_j-r^2_j}$ is the difference of the radii between two circles. Therefore, $d_1$ may be understood as the amount of energy to deform $(\cv_2,\rv^2)$ onto $(\cv_1,\rv^1)$. Likewise, $d_2$ is the amount of energy to deform $(-\cv_2, \pi-\rv^2)$  onto $(\cv_1,\rv^1)$, so that $d([p_1],[p_2])$ is the minimal amount of energy to deform $[p_2]$ onto $[p_1]$. The following result is required for the asymptotic theory in the next section.

\begin{lem}\label{lem:metric}
The distance function $d$ on $P$, or on $P_0$, is a metric or a pseudo-metric, respectively.
\end{lem}

A proof is given in Section~\ref{sec:technic}.

\subsection{Consistency and asymptotic normality}

A consistency of the estimator $\hat{p} = (\hat{\cv}, \hat{\r}) \in P_0$ with a parameter $p \in P_0$ can be evaluated by showing that $\hat{p}$ approaches $p$ with respect to the distance $d$ as $n \to \infty$, which ensures that $\lim_{n\to\infty}d([\hat{p}], [p]) = 0$. On the other hand, there are at least two global minimizers in $P_0$ because of the equivalence relation. Moreover, there may be non-unique solutions to the Fr\'{e}chet $\rho$-means of (\ref{eq:minimizerModel}) and (\ref{eq:estimation}), even in $P$. To accommodate such general situations, define the sets of minimizers of  (\ref{eq:minimizerModel})  and (\ref{eq:estimation}) as follows.

The population Fr\'{e}chet $\rho$-means of $X$ in $P_0$ is denoted by
$$E = \left\{   \mu \in P_0   :   \E\{\rho^2(X,\mu)\} = \inf_{p\in P_0} \E\{\rho^2(X,p)\}\right\} \subset P_0,
$$
and the set of sample Fr\'{e}chet $\rho$-means of $X$ is denoted by
$$E_n =\left\{  \mu \in P_0   : \sum_{j=1}^n \rho^2(X_j,\mu) = \inf_{p\in P_0} \sum_{j=1}^n\rho^2(X_j,p) \right\} \subset P_0.
$$
 Let $[E] = \{[p]: p \in E\}$ and  $[E_n] = \{[p]: p \in E_n\}$. The consistency of $E_n$ will be determined based on the distance between elements in $E_n$ and in $E$, measured by $d$.

In considering an asymptotic normality, we shall assume the population Fr\'{e}chet $\rho$-mean $[\mu] \in [E]$ is unique. In that case, the population Fr\'{e}chet $\rho$-mean set $E$ has precisely two elements, which will be handled by choosing one element of $E$. Since both $P$ and $P_0$ are not vector spaces, the usual normal distribution is not defined on such spaces. However, for any point $\mu \in P_0$, there exists a local chart $(\phi, U)$ that locally parameterize a small neighborhood $ A_\mu \subset P_0$ of $\mu \in A_\mu$, i.e., for some open set $U \subset \Real^{\nu}, \nu = m+K$, $\phi(p) \in U$ for $p \in A_\mu$.

In fact, $P_0 = S^m \times (0,\pi)^{K}$ is a smooth Riemannian manifold with intrinsic dimension $\nu = m +K$, since it is a product of $S^m$ and an open interval in Euclidean space. Then $P_0$ is naturally embedded into $\Real^{\nu+1}$, since $S^m \subset \Real^{m+1}$ and $(0,\pi)^K \subset \Real^K$. For a point $p = (\cv,\rv) \in P_0 \subset \Real^{\nu+1}$, let $T_pP_0$ be the affine $\nu$-dimensional hyperplain tangent to $P_0$ at $\mu$, which is the direct product of the spaces tangent to $S^m$ and $(0,\pi)^K$:
$$T_pP_0 = T_{(\cv,\rv)}P_0 = T_\cv S^m \times T_{\r}(0,\pi)^K.$$
Precisely, we consider the tangent space $T_cS^m$ of $S^m$ at $c \in S^m$ as the parametrization of the real tangent hyperplain of $S^m$ to $\Real^m$. That is, $T_\cv S^m \cong \Real^m$. Let $\cv  = \ev_{m+1}$, then the exponential map $\mbox{Exp}_\cv: \Real^m \to S^m$ is defined for $\vv_1 \in \Real^m$ by
$$\mbox{Exp}_\cv (\vv_1) = \left( \frac{\vv_1}{\norm{\vv_1}}\sin\norm{\vv_1} , \cos\norm{\vv_1} \right),$$
with a convention of $\mbox{Exp}_c(0) = c$.
Denote $\v = (\v_1,\v_2) \in T_pP_0$ for $\v_1 \in T_\cv S^m \cong \Real^m$, $\v_2 \in T_{\r}(0,\pi)^K = \Real^K$.

An example of the local chart $(\phi,U)$ is therefore the pair of the inverse exponential map and the tangent space. The exponential map at $p = (\cv ,{\r}) \in P_0$ is a map from the tangent space $T_{(\cv,\rv)}P_0$ at $p = (\cv,\rv)$ given by
$$ \mbox{exp}_{(\cv,\rv)}(\v) = ( \mbox{Exp}_\cv (\v_1), \r+\v_2 ),  \quad \v = (\v_1,\v_2) \in T_{(\cv,\rv)}P_0, $$
Therefore for any $p = (\cv ,{\r}) \in P_0$, there is an open set
$$U = \{\u \in \Real^m : \norm{\u} < \pi/2 \} \times (-\r,\pi-\r)^K \in T_pP_0$$
and $\phi = \mbox{exp}_{(\cv,\rv)}^{-1}$.
A version of central limit theorem will be developed on the $U$ or the tangent space at $\mu$, which is a local approximation of $P_0$.
The inverse of the exponential map will be used to map $p \in P_0$ to $T_{\mu}P_0$ for some $p$ near $\mu$.

In a local chart $(\phi,U)$ of $P_0$ near $\mu = \phi^{-1}(0)$, denote by $\nabla \rho^2\{x,\phi^{-1}(u)\}$ the gradient function of $\rho^2\{\cdot,\phi^{-1}(\cdot)\}$ in the second argument,
and denote by $H\rho^2\{x,\phi^{-1}(u)\}$ the Hessian matrix of the second order derivatives.
An assumption we require is that the random element $X = (X_{(1)},\ldots,X_{(n)})$ does not degenerate, i.e., for each $j$, the geodesic variance of $X_{(j)}$ is non-zero.
\begin{thm}\label{thm:consistencyandnormality}
Suppose the distribution of $X$ satisfies (A1) when $\rho_I$, $\rho_E$ or $\rho_N$ are used. The assumption (A1) is not required when $\rho_S$ is used. If the distribution of $X$ does not degenerate, then
\begin{enumerate}
  \item[i)] $E$ exists;
  \item[ii)] $E_n$ is strongly consistency with $E$ in the sense that
\begin{equation}\label{eq:consistency}
    \lim_{n\to\infty} \max_{p_n \in E_n} \min_{p\in E} d(p,p_n)  = 0 \ \mbox{ with probability 1},
\end{equation}
\end{enumerate}
and if additionally $[E] = \{[\mu]\}$ is unique,
then
\begin{enumerate}
\item[iii)] there exists a sequence $\mu_n \in E_n$ such that $\lim_{n\to\infty}d(\mu_n,\mu) = 0$ almost surely for a fixed $\mu \in E$, and that for any local chart $(\phi,U)$ near $\mu = \phi^{-1}(0)$,
there exist a $\nu \times \nu$ matrix $A_\phi$ and a $\nu \times \nu$ covariance matrix $\Sigma_\phi$  such that
 \begin{equation}\label{eq:asympnormality}
\sqrt{n} A_\phi \{\phi(\mu_n) -  \phi(\mu)\} \To N_\nu(0, \Sigma_\phi) \mbox{  in distribution as } n \to \infty.
\end{equation}
In particular, the matrices above are given by
$A_\phi = \mathbb{E}\{H\rho^2(X,\mu)\}$ and $\Sigma_\phi = \Cov\{\nabla \rho^2(X,\mu)\}$.
\end{enumerate}
\end{thm}

Note that the theorem requires a minimal assumption, namely the observations are i.i.d. and non-degenerate. This gives a much flexibility in modeling the error distribution $\epsilon_j$ across different $j$s.

To focus on the estimator $\hat{\cv}$ of the axis of rotation $c$, let $U_1 = \{\u \in \Real^m : \norm{\u} < \pi/2 \} \subset T_cS^m$ be the first $m$-coordinates of $U$, and $\phi_1$ be the first $m$ elements of $\phi$. Using the inverse exponential map and the tangent space, $\phi_1: S^m \to T_cS^m \cong \Real^m$, $\phi_1(x) = \mbox{Exp}^{-1}_c (x)$ is a mapping from a neighborhood of $c \in S^m$ to $U_1$.

The estimator $\hat{\cv}$ from (\ref{eq:estimation}) is found by minimizing sum of squared errors over  $n$ different samples and also over $K$ different directions. In the model used in \cite{Schulz2012},
\begin{equation}\label{eq:model1}
  X_j = R(\cv,\theta)\mu_j\oplus \epsilon_j \ \ (j=1,\ldots, K),
\end{equation}
the number of directions $K$ has a similar role as the sample size. The following corollary shows the variance of the estimator is smaller for larger number of $K$.

\begin{cor}\label{cor:CLT}
Suppose the conditions in Thoerem~\ref{thm:consistencyandnormality} are satisfied. In addition suppose that the marginal distribution of $\Xv$, $X^{(j)}$, satisfies
$X^{(j)}= R(\cv,\theta)\yv_j \oplus \epsilon_j$ for each $j = 1,\ldots,K$, as in (\ref{eq:model1}), and $\epsilon_j$'s are i.i.d. Then
for $\hat{\cv}_n$ such that $\mu_n = (\hat{\cv}_n,r_n)$,
there exist an $m\times m$ matrix $\bar{A}_{\phi_1}$ and an $m\times m$ covariance matrix $\bar\Sigma_{\phi_1}$ such that
 \begin{equation*}
\sqrt{nK} \bar{A}_{\phi_1} \{\phi_1(\hat{\cv}_n) -  \phi_1(c)\} \To N_\nu(0, \bar\Sigma_{\phi_1}) \mbox{  in distribution as } n \to \infty, K \to \infty.
\end{equation*}
\end{cor}

Note that the limits are applied sequentially. That is, the large sample assumption ($n\to\infty$) remains to be the major driver for the asymptotic normality.

Technical details can be found in Section~\ref{sec:technic}.

\begin{remark}
We have  used the notations and theories developed for Fr\'{e}chet mean on general pseudo-metric spaces \citep{Ziezold1977,Bhattacharya2003,Huckemann2011,Huckemann2011CLT}. These theories were developed mainly for shape spaces \citep{Dryden1998}, but the applications are much broader than the shape space, as this paper exemplifies. We also like to mention that there might be an alternative approach in investigation of asymptotic properties of the proposed estimator. For example, the work of \cite{Chang2001} concerns a general M-estimation for Stiefel manifolds, and our estimation procedure can be understood as an M-estimation in a direct product of Stiefel manifolds.
\end{remark}

%
%

\section{Technical details}\label{sec:technic}
\subsection{Proof of Lemma~\ref{lem:metric}}
\begin{proof}[Proof of Lemma~\ref{lem:metric}]
The non-negativity and symmetry of $d$ in both $P$ and $P_0$ are immediate. It is also easy to see that $d([p_1],[p_2]) = 0 $ if and only if $[p_1] = [p_2]$, but $d(p_1,p_2) = 0$ if $p_1=(\cv,\rv) \neq (-\cv, \pi-\r) = p_2$. The proof is completed by a triangle inequality which we provide for $(P_0,d)$ in the following.

First note that $\arccos(\cv_1'\cv_2)$ is the length of the shortest great circle segment connecting $\cv_1$ and $\cv_2$. For any $\cv_3 \in S^m$, the three points $\cv_1,\cv_2$ and $\cv_3$ and the sides given by the great circle segments form a spherical triangle on $S^m$. Therefore by the triangle inequality for spherical triangles \citep[p. 17]{Ramsay1995}, we have
\begin{equation}\label{eq:sphericaltriangleinequality}
    \arccos(\cv_1'\cv_2) \le \arccos(\cv_1'\cv_3) + \arccos(\cv_3'\cv_2).
\end{equation}

Without loss of generality, assume that $d_1(p_1,p_2) \le d_2(p_1,p_2)$.
For any $p_3 = (\cv_3,r^3) \in P_0$, we get, by (\ref{eq:sphericaltriangleinequality}),
\begin{eqnarray}\label{eq:triangleinequality1}
  d_1(p_1,p_2) &=& \sqrt{\arccos(\cv_1'\cv_2)^2 + \|r^1-r^3+r^3-r^2\|^2  }  \nonumber \\
             &\le&  \sqrt{\{\arccos(\cv_1'\cv_3) + \arccos(\cv_3'\cv_2)\}^2 + \{\|r^1-r^3\|+\|r^3-r^2\|\}^2  }  \nonumber \\
             &\le& \sqrt{\arccos(\cv_1'\cv_3)^2+ \|r^1-r^3\|^2}+ \sqrt{\arccos(\cv_3'\cv_2)^2+ \|r^3-r^2\|^2}  \nonumber  \\
             &= & d_1(p_1,p_3) + d_1(p_3,p_2),
\end{eqnarray}
and
 \begin{eqnarray}
  d_1(p_1,p_2) &\le&  d_2(p_1,p_2)  \nonumber \\
              &=& \sqrt{\arccos\{\cv_1'(-\cv_2)\}^2 + \|r^1- (\pi-r^2)\|^2  }  \nonumber \\
             &\le&  \sqrt{\{\arccos(\cv_1'\cv_3) + \arccos(-\cv_3'\cv_2)\}^2 + \{\|r^1-r^3\|+\|\pi-r^3-r^2\|\}^2  }  \nonumber \\
             &= & d_1(p_1,p_3) + d_2(p_3,p_2). \label{eq:triangleinequality2}
\end{eqnarray}
Similarly,
\begin{eqnarray}
   d_1(p_1,p_2) &\le&  d_2(p_1,p_3) + d_1(p_3,p_2), \nonumber \\
  d_1(p_1,p_2) &\le&  d_2(p_1,p_3) + d_2(p_3,p_2). \label{eq:triangleinequality4}
\end{eqnarray}

Combining (\ref{eq:triangleinequality1}-\ref{eq:triangleinequality4}),
$$d(p_1,p_2) = d_1(p_1,p_2) \le \min\{d_1(p_1,p_3), d_2(p_1,p_3)\} +
                                \min\{d_1(p_3,p_2), d_2(p_3,p_2)\},$$
which proves the triangle inequality.
\end{proof}

\subsection{Proof of the main results}

\begin{proof}[Proof of Thoerem~\ref{thm:consistencyandnormality}]
We first prove that i) $E$ exists, in the special case of $\rho = \rho_I$. The proof for the cases $\rho = \rho_E, \rho_S, \rho_N$ is similar, and is omitted.
Let
$$F(\cv,r) = \mathbb{E}[\rho^2\{X, (\cv,\rv)\}]  = \frac{1}{K}\sum_{j=1}^K \mathbb{E}\{\arccos(X_{(j)}'\cv)-r_j\}^2 = \sum_{j=1}^K F_j(\cv,r_j),$$
where $X_{(j)} \in S^m$ the $j$th marginal random element of
$X \in (S^m)^K$ and $\r = (r_1,\ldots, r_K)$.
Then Fr\'{e}chet $\rho$-means of $X$ is
$$E = \{(\cv,\rv) \in P_0 | F(\cv,\rv) = \inf_{\cv_0,\r_0} F(\cv_0,\r_0)\}.$$
Consider the closer of $P_0$, $\overline{P_0} = S^m \times [0,\pi]$, which is compact. Since $F$ is continuous, there exists a $(\cv_1,\r^1) \in \overline{P_0}$ such that $F(\cv_1,\r^1) = \inf_{\cv_0,\r_0} F(\cv_0,\r_0)\}$.
The set $E$ is nonempty if such $\r^1$ satisfies $r^1_j \in (0,\pi)$ for all $j$. For each $j$, since $F_j(\cv,r_j) = F(-\cv,\pi-r_j)$, $r_j >0$ if and only if $r_j < \pi$.
Since the distribution of $X$ does not degenerate, for any $\cv \in S^m$ and for any $1 \le j \le K$, there exists $\epsilon>0$ such that $P(\arccos(X_{(j)}'\cv)>\epsilon)>0$, which leads to
\begin{align*}
  \mathbb{E}(\arccos(X_{(j)}'\cv))
 & = \mathbb{E}(\arccos(X_{(j)}'\cv)1_{0\le \arccos(X_{(j)}'\cv)\le\epsilon}) + \mathbb{E}(\arccos(X_{(j)}'\cv)1_{\arccos(X_{(j)}'\cv)> \epsilon}) \\
& > \epsilon P(\arccos(X_{(j)}'\cv)>\epsilon) > 0.
\end{align*}
For any $\cv \in S^m$, $F_j(\cv,0) = \mathbb{E}\{\arccos^2(X_{(j)}'\cv)\} > \mathbb{E}[\{\arccos(X_{(j)}'\cv)- \mathbb{E}(\arccos(X_{(j)}'\cv))\}^2]$. Therefore,
 $F(\cv,0) = \sum_{j=1}^K F_j(\cv,0) > \inf_{\cv_0,\r_0} F(\cv_0,\r_0)$, from which we conclude that $E$ is nonempty.

A proof for strong consistency (ii) is based on the arguments in \cite{Huckemann2011}, and we work with the following two general definitions.

\begin{defn}
Let $E_n$ be a random closed set and $E$ be a deterministic closed set in $(P_0,d)$. We then say
\begin{description}
  \item[(ZC)] $E_n$ is a strong consistent estimator of $E$ in the sense of Ziezold \citep{Ziezold1977} if
$$ \bigcap_{n=1}^\infty \overline{\bigcup_{k=n}^\infty E_k} \subset E \quad \mbox{almost surely.}$$
  \item[(BPC)] $E_n$ is a strong consistent estimator of $E$ in the sense of Bhattacharya and Patrangenaru \citep{Bhattacharya2003} if $E$ is non-empty and
   for every $\epsilon>0$, there is a sufficiently large $n$ such that
   $$ \bigcup_{k=n}^\infty E_k \subset \{ p \in P_0: d(E,p) \le \epsilon \}  \quad \mbox{almost surely.}$$
\end{description}
\end{defn}

If $E_n$ is a Ziezold-consistent (ZC) estimator of $E$, then for any sequence $\mu_n \in E_n$ that converges to some $p \in P_0$, then such $p$ is in $E$ with probability 1.
If Bhattacharya--Patrangenaru consistency (BPC) holds for $E_n$, then for sufficiently large $n$ all elements of $E_n$  are in the $\epsilon$-neighborhood of $E$ almost surely.

\begin{remark}
The BPC is a stronger property than ZC as we shall see in the following theorem. As a simple example, suppose each $E$ and $E_n$ have only one element denoted by $\mu$ and $\mu_n$. Then BPC holds if and only if $\mu_n \to \mu$ almost surely, which coincides to the usual notion of strong consistency. On the other hand, let $\mu_n = n$ diverges, then $\bigcap_{n=1}^\infty \overline{\bigcup_{k=n}^\infty E_k} = \emptyset$ so the Ziezold consistency holds, which is somewhat less meaningful. The following theorem states that BPC and ZC are the same when such diverging $E_n$ is prevented.
\end{remark}

\begin{thm}\label{thm:ZC-BPC}
Suppose that $E$ is non-empty and $d$ is a quasi-metric.

1. BPC implies ZC.

2. ZC implies BPC if a) $P_0$ is totally bounded or if b)
\begin{enumerate}
  \item[i)] $E(\rho(X,p)^2) <\infty $ for all $p \in P_0$;
  \item[ii)] $(P_0,d)$ satisfies the Heine-Borel property;
  \item[iii)] $\rho(x,p)$ is growing at the extremes of $P$. More precisely, if $d(p,p_n ) \to \infty$ for some $p, p_n \in P_0$, then for all $x \in Q$ such that $\rho(x,p)< B <\infty$, there is an increasing sequence $M_n \to \infty$ satisfying $\rho(x,p_n) \ge M_n$.
\end{enumerate}
\end{thm}

\begin{proof}[Proof of Theorem \ref{thm:ZC-BPC}]
1. Let $B_n = \cup_{k=n}^\infty E_k$. Then since $E_n$ is BP-consistent, given any $\epsilon >0$, there is  sufficiently large $n$ such that $B_n$ is in the $\epsilon$-neighborhood of $E$, and thus for any $b_n \in B_n$, $d(b_n, E) \le \epsilon$. Let $B = \cap_{n=1}^\infty \overline{B_n} $. By the definition of the closure, For any $b \in B$, $b \in \overline{B_n}$ for all $n$ and  there exists a sequence $b_n \in B_n$ such that $b_n \to b$ as $n \to \infty$.

Given $b \in B$, for any  arbitrarily small  $\epsilon>0$, one can choose $b_N \in B_N$ for sufficiently large $N$ satisfying $d(b_N,E) \le \epsilon$ and $d(b,b_N) < \epsilon$. Therefore,
$ d(b,E) = \inf_{e\in E} d(b,e) \le d(b,b_N) + \inf_{e\in E} d(b_N,e) \le 2 \epsilon$. Since $E$ is closed, letting $\epsilon \to 0$ shows that $b \in E$, thus leads to Ziezold Consistency.

2. Consider a sequence $p_n \in E_n$ determined by
$$d(p_n,E) = \max_{p \in E_n} d(p,E) = r_n.$$
Then either $r_n \to 0$ (satisfying BPC) or $r_n$ does not converge to 0. If $r_n \nrightarrow 0$, there is a sequence $n(k)$ such that $r_{n(k)} \ge r_0 > 0$, and if there is an accumulation point of $p_{n(k)}$, the accumulation point has a positive distance to $E$, which is a contradiction to ZC. So, whenever $r_n \nrightarrow 0$, there should be no accumulation point. We will rule out this case by contradiction for each of conditions (a) or (b).

a) If $P_0$ is totally bounded, for any small $\epsilon>0$ and a finite cover $\{A_j\}$ such that $P_0 \subset \cup A_j $ and $\mbox{diam}(A_j) = \epsilon$, there is only finitely many $p_{n(k)}$ in each $A_j$. This is a contradiction to the existence of the subsequence. Thus $r_n \to 0$. Bhattacharya and Patrangenaru have noted a similar observation for a less general definition of Fr\'{e}chet mean \citep{Bhattacharya2003}.

b) Since $(P,d)$ satisfies the Heine-Borel property, $r_n = d(p_n,E)$ is unbounded, because otherwise $p_{n(k)}$ is bounded and there exists an accumulation point of $p_{n(k)}$. Therefore, $\limsup r_{n} = \infty$, and let $l_n$ be the subsequence satisfying $\lim r_{l_n} = \infty$.

Now assume $\mathbb{E}\{\rho(X,p)^2\} <\infty $ for all $p \in P_0$. Then there exist $p_0 \in P_0, C >0$ such that $P\{\rho(X,p_0) < C\}0 > 0$.
Suppose otherwise that for any $p_0$ and $C$, $P\{\rho(X,p_0) < C\} = 0$, which is the same as  $P\{\rho(X,p_0) \ge C\} = 1$. Then
$\mathbb{E}\{\rho(X,p_0)^2\} \ge C^2$, which is a contradiction $\mathbb{E}\{\rho(X,p_0)^2\} <\infty $ since $C$ is arbitrarily large.

By above argument, we choose $p_0 \in P_0$ and $C>0$ such that $P\{\rho(X,p_0) < C\} > 0$. Then since $X_i$'s are i.i.d., there is a subsequence $k(n)$ of $n$, and for each $n$ there is a subsequence $j_{1},\ldots,j_{k(n)}$ of ${1,\ldots, n}$ satisfying $\rho(X_{j_i},p_0) <C$ for all $i = 1,\ldots, {k(n)}$, a.s., and
$$\frac{k(n)}{n} \to P\{\rho(X,p_0)<C\} > 0.$$

Then by assumption (iii),  for $n \in \{l_i: i =1,\ldots, n \}$,
$$
\ell_n = F_n(p_n)  = \frac{1}{n}\sum_{i=1}^n \rho(X_i,p_n)^2 \ge \frac{1}{n}\sum_{i=1}^{k(n)} \rho(X_{j_i},p_n)^2 >  \frac{k(n)}{n} M_n^2
$$
so $ \limsup {\ell_n} = \infty$ almost surely.
Meanwhile, for any $p \in P_0$, a strong law of large numbers yields $\ell_n \le F_n(p) = \frac{1}{n}\sum_{i=1}^n \rho(X_i,p)^2 \to F(p) = \mathbb{E}\{\rho(X,p)^2\} < \infty $ almost surely, which is a contradiction.
\end{proof}

The Ziezold consistency of $E_n$ is then shown by an application of  Theorem A.3 of  \cite{Huckemann2011}. In particular, since the support of $X$, $(S^m)^{K}$, is compact, $\rho$ is continuous and $(P_0,d)$ is a separable metric space. These conditions satisfy the assumptions of Theorem A.3 of  \cite{Huckemann2011}, which gives the Ziezold consistency of $E_n$. Moreover, since $P$ is totally bounded, by Theorem \ref{thm:ZC-BPC}, we get (\ref{eq:consistency}).

The asymptotic normality (\ref{eq:asympnormality})  is an application of Theorem 6 of \cite[p.444]{Huckemann2011CLT}, the assumptions of which are justified provided that $(S^m)^K$ is compact, and $\rho^2$ is smooth in terms of the second argument.
\end{proof}

\begin{proof}[Proof of Corollary~\ref{cor:CLT}]
We use notations that $\phi(\mu) = (\phi_1(c), \phi_2(r))$, for $\mu = (\cv,r)$, and that $\phi^{-1}_2(x_{m+1},\ldots,x_{m+K}) = [r_1(x_{m+1}), \ldots, r_K(x_{m+K})]$. Then,
following the definition (\ref{eq:defofrho}) of $\rho$,  write
$$
\nabla\rho^2(X,\phi^{-1}(x))
    = \sum_{j=1}^K \nabla d_g^2[ \delta\{ \phi_1^{-1}(x_1,\ldots,x_m), r_j(x_{m+j})\}, X^{(j)}]
    = \begin{pmatrix}
        \sum_{j=1}^K g_1^j(x,X^{(j)}) \\
        \vdots \\
        \sum_{j=1}^K g_m^j(x,X^{(j)})  \\
        g_{m+1}^1(x, X^{(1)}) \\
        \vdots\\
        g_{m+K}^K(x, X^{(K)}) \\
      \end{pmatrix},
$$
where $g_i^j(x,X^{(j)}) = \frac{d}{dx_{i}} d_g^2[ \delta\{ \phi_1^{-1}(x_1,\ldots,x_m), r_j(x_{m+j})\}, X^{(j)}]$ for $i=1,\ldots,m+K, j = 1,\ldots, K$.
In particular, when $x = 0$,
$$g_i^j(0,X^{(j)}) = \left\{
                       \begin{array}{ll}
                         2 \{\arccos(\cv' R(c,\theta) \mu_j \oplus \epsilon_j) - r_j\} \frac{d}{dx_i} \arccos(\cv' R(c,\theta) \mu_j \oplus \epsilon_j), & {i = 1,\ldots, m;} \\
                         2 \{\arccos(\cv' R(c,\theta) \mu_j \oplus \epsilon_j) - r_j\}\frac{d}{dx_i} r_j, & {i =m+1,\ldots,m+K.}
                       \end{array}
                     \right.
$$
Since $\cv' R(c,\theta) \mu_j = \cv'\mu_j$, the $g_i^j(0,X^{(j)})$ depends only on $\epsilon_j$, but not on $\theta$. This and the fact that $\mu \oplus e_j$'s are i.i.d. lead to that $\{g_i^j(0,X^{(j)}) \}_{j=1\ldots,K}$ are mutually independent.
Since $\E \{ \nabla\rho^2(X,\mu)\} = 0$, we have
\begin{equation}\label{eq:cor-grad}
\Sigma_\phi = \Cov \{ \nabla\rho^2(X,\mu)\} = \E [\nabla\rho^2(X,\mu) \{\nabla\rho^2(X,\mu)\}']
 =  \begin{pmatrix}
    K \Sigma_{\phi_1} & \Sigma_{12}  \\
      \Sigma_{12}' & \Sigma_{22}  \\
    \end{pmatrix},
\end{equation}
where the $(i,l)$th element of the $m \times m$ matrix $\Sigma_{\phi_1}$ is $\E K^{-1} \sum_{j=1}^K g_i(0, X^{(j)}) g_l(0, X^{(j)})$,
       the $(i,l)$th element of $\Sigma_{12}$ is
       $\E g_i(0, X^{(l)}) g_{m+l}(0, X^{(l)})$
        for $ i = 1,\ldots,m ,l = 1,\ldots,K$,
        and $\Sigma_{22} = \diag_{l=1,\ldots,K} \E \{g_{m+l}(0, X^{(l)})\}^2$,
        where $\diag_{l=1,\ldots,K} a_l$ is a $K\times K$ diagonal matrix with $a_l$ being the $l$th diagonal entry.
Note that since $\rho$ is smooth, the first and second moments of $(g_i(0, X^{(j)})g_l(0, X^{(j)})$ exist for each $j = 1,\ldots, K$. Thus, by a law of large numbers, there exists a $\bar\Sigma_{\phi_1}$ such that $\Sigma_{\phi_1} \to \bar\Sigma_{\phi_1}$ as $K \to \infty$.

For the Hessian, we have
\begin{equation}\label{eq:cor-Hessian}
\E \{ H\rho^2(X,\mu)\} = \E \{ H\rho^2(X,\phi^{1}(0) )\}
 =  \begin{pmatrix}
      K A_{\phi_1}  & A_{12}  \\
      A_{12}' & A_{22}  \\
    \end{pmatrix},
\end{equation}
where  $g_{il}(x,X^{(j)})  = \frac{d^2}{dx_{l}dx_{i}}  d_g^2[ \delta\{ \phi_1^{-1}(x_1,\ldots,x_m), r_j(x_{m+j})\},  X^{(j)}]$ for $i,l=1,\ldots,m+K, j = 1,\ldots, K$,
\begin{align*}
A_{\phi_1} &= (\E K^{-1} \sum_{j=1}^K g_{il}(0,X^{(j)}))_{i,l = 1,\ldots,m},\\
    A_{12} &= (\E g_{i,m+l}(0,X^{(l)}))_{i,l= 1,\ldots,m}, \\
    A_{22} &= \diag_{l = 1,\ldots,K} \E g_{m+l,m+l}(0,X^{(l)}).
\end{align*}
Similar to the gradient case, a law of large numbers ensures that there exists a  $\bar{A}_{\phi_1}$ such that $A_{\phi_1} \to \bar{A}_{\phi_1}$ as $K \to \infty$.
The result of Theorem~\ref{thm:consistencyandnormality} leads to
$$
\sqrt{n}
\begin{pmatrix}
      K A_{\phi_1}  & A_{12}  \\
      A_{12}' & A_{22}  \\
    \end{pmatrix}
\left( \begin{pmatrix}
      \phi_1(\mu_n)    \\
      \phi_2(\mu_n)    \\
    \end{pmatrix}
    -  \begin{pmatrix}
      \phi_1(\mu)    \\
      \phi_2(\mu)    \\
    \end{pmatrix}
       \right) \to N_\nu (0, \Sigma_\phi)
$$
in distribution as $n \to \infty$, for any $K$.
Taking the marginal distribution for the first $m$ dimensions, and multiplying $1/\sqrt{K}$, we get,
$$\sqrt{nK} A_{\phi_1} ( \phi_1(\hat{\cv}_n) -  \phi_1(c) ) + \frac{\sqrt{n}}{\sqrt{K}}A_{12}(\phi_2(\mu_n)-\phi_2(\mu)  ) \to N_\nu (0, \Sigma_{\phi_1}).$$
The proof is completed by letting $K \to \infty$ and by noting that the second term in the left hand side is $O_p(\sqrt{K})$.

\end{proof}

\bibliographystyle{asa}
\bibliography{PNS}

\end{document}